\documentclass{amsart}%
\usepackage{amsmath}
\usepackage{amssymb}
\usepackage{array}
\usepackage{amsthm}
\usepackage{amsfonts}
\usepackage{amssymb}
\usepackage{graphicx}%
\providecommand{\U}[1]{\protect\rule{.1in}{.1in}}
\providecommand{\U}[1]{\protect \rule{.1in}{.1in}}
\providecommand{\U}[1]{\protect \rule{.1in}{.1in}}
\newtheorem{theorem}{Theorem}
\theoremstyle{plain}

\newtheorem{corollary}{Corollary}

\newtheorem{definition}{Definition}
\newtheorem{example}{Example}

\newtheorem{lemma}{Lemma}

\newtheorem{proposition}{Proposition}

\numberwithin{equation}{section}

\begin{document}
\title{\textit{Nearness Rings on Nearness Approximation Spaces*}}
\thanks{* This paper is a part of Ebubekir \.{I}nan's PhD thesis which approved on 20.03.2015 by \.{I}n\"{o}n\"{u} University Graduate School of Natural and Applied Sciences, T\"{u}rkiye.}
\author{MEHMET AL\.{I} \"{O}ZT\"{U}RK}
\address{{\small \textit{Department of Mathematics}}\\
{\small \textit{Faculty of Arts and Sciences}}\\
{\small \textit{Ad\i yaman University}}\\
{\small \textit{Ad\i yaman, T\"{u}rkiye}}}
\email{maozturk@adiyaman.edu.tr}

\author{EBUBEK\.{I}R \.{I}NAN}
\address{{\small \textit{Department of Mathematics}}\\
{\small \textit{Faculty of Arts and Sciences}}\\
{\small \textit{Ad\i yaman University}}\\
{\small \textit{Ad\i yaman, T\"{u}rkiye}}}
\email{einan@adiyaman.edu.tr}


\begin{abstract}
In this paper, we consider the problem of how to establish algebraic
structures on nearness approximation spaces. Essentially, our approach is to
define the nearness ring, nearness ideal and nearness ring of all weak cosets
by considering new operations on the set of all weak cosets. Afterwards, our
aim is to study nearness homomorphism on nearness approximation spaces, and to
investigate some properties of nearness rings and ideals. 
\end{abstract}

\maketitle

\section{Introduction}

Nearness approximation spaces and near sets were introduced in 2007 as a
generalization of rough set theory \cite{Peters1, Peters}. More recent work
consider generalized approach theory in the study of the nearness of non-empty
sets that resemble each other \cite{Peters5} and a topological framework for
the study of nearness and apartness of sets \cite{Naimpally}. An algebraic
approach of rough sets has been given by Iwinski \cite{Iwinski}. Afterwards,
rough subgroups were introduced by Biswas and Nanda \cite{Biswas}. In 2004 Davvaz investigated the concept of roughness of rings
\cite{Davvaz2004} (and other algebraic approaches of rough sets in
\cite{Yamak, Rasouli}).

Near set theory begins with the selection of probe functions that provide a
basis for describing and discerning affinities between objects in distinct
perceptual granules. A probe function is a real-valued function representing a
feature of physical objects such as images or collections of artificial
organisms, e.g. robot societies.

In the concept of ordinary algebraic structures, such a structure that
consists of a nonempty set of abstract points with one or more binary
operations, which are required to satisfy certain axioms. For example, a
groupoid is an algebraic structure $\left(  A,\circ\right)  $ consisting of a
nonempty set $A$ and a binary operation \textquotedblleft$\circ$%
\textquotedblright\ defined on $A$ \cite{Clifford1961}. In a groupoid, the
binary operation \textquotedblleft$\circ$\textquotedblright\ must be only
closed in $A$, i.e., for all $a,b$ in $A$, the result of the operation $a\circ
b$ is also in $A$. As for the nearness approximation space, the sets are
composed of perceptual objects (non-abstract points) instead of abstract
points. Perceptual objects are points that have features. And these points
describable with feature vectors in nearness approximation spaces
\cite{Peters}. Upper approximation of a nonempty set is obtained by using the
set of objects composed by the nearness approximation space together with
matching objects. In the algebraic structures constructed on nearness
approximation spaces, the basic tool is consideration of upper approximations
of the subsets of perceptual objects. In a groupoid $A$ on nearness
approximation space, the binary operation \textquotedblleft$\circ
$\textquotedblright\ may be closed in upper approximation of $A$, i.e., for
all $a,b$ in $A$, $a\circ b$ is in upper approximation of $A$.

There are two important differences between ordinary algebraic structures and
nearness algebraic structures. The first one is working with non-abstract
points while the second one is considering of upper approximations of the
subsets of perceptual objects for the closeness of binary operations.

In 2012, E. \.{I}nan and M. A. \"{O}zt\"{u}rk \cite{Inan, Inan2014}
investigated the concept of near groups on nearness approximation spaces.
Moreover, in 2013, M. A. \"{O}zt\"{u}rk at all \cite{Ozturk2013} introduced
near group of weak cosets on nearness approximation spaces. And in 2015, E. \.{I}nan and M. A. \"{O}zt\"{u}rk \cite{Inansemi}
investigated the nearness semigroups.  In this paper, we
consider the problem of how to establish and improve algebraic structures of
nearness approximation spaces. Essentially, our aim is to obtain algebraic
structures such as nearness rings using sets and operations that ordinary are
not being algebraic structures. Moreover, we define the nearness ring of all
weak cosets by considering operations on the set of all weak cosets. To define
this quotient structure we don't need to consider ideals.

\section{Preliminaries}

\subsection{Nearness Approximation Spaces \cite{Peters}}

Perceptual objects are points that describable with feature vectors. Let
$\mathcal{O}$ be a set of perceptual objects. An object description is defined
by means of a tuple of function values $\Phi\left(  x\right)  $ associated
with an object $x\in X\subseteq\mathcal{O}$. The important thing to notice is
the choice of functions $\varphi_{i}\in B$ used to describe any object of
interest. Assume that $B\subseteq%
%TCIMACRO{\tciFourier}%
%BeginExpansion
\mathcal{F}%
%EndExpansion
$ is a given set of functions representing features of sample objects
$X\subseteq\mathcal{O}$. Let $\varphi_{i}\in B$, where $\varphi_{i}%
:\mathcal{O}\longrightarrow%
%TCIMACRO{\U{211d} }%
%BeginExpansion
\mathbb{R}
%EndExpansion
$. In combination, the functions representing object features provide a basis
for an object description $\Phi:\mathcal{O}\longrightarrow%
%TCIMACRO{\U{211d} }%
%BeginExpansion
\mathbb{R}
%EndExpansion
^{L}$, a vector containing measurements (returned values) associated with each
functional value $\varphi_{i}\left(  x\right)  $, where the description length
is $\left\vert \Phi\right\vert =L$.

\small

Object Description: $\Phi\left(  x\right)  =\left(  \varphi_{1}\left(
x\right)  ,\varphi_{2}\left(  x\right)  ,\varphi_{3}\left(  x\right)
,...,\varphi_{i}\left(  x\right)  ,...,\varphi_{L}\left(  x\right)  \right)  $.

\normalsize
Sample objects $X\subseteq\mathcal{O}$ are near to each other if and only if
the objects have similar descriptions. Recall that each $\varphi$ defines a
description of an object. Then let $\Delta_{\varphi_{i}}$ denote 
$\Delta_{\varphi_{i}}=\left\vert \varphi_{i}\left(  x^{\prime}\right)
-\varphi_{i}\left(  x\right)  \right\vert $, where $x,x^{\prime}\in\mathcal{O}$. The difference $\Delta_{\varphi}$ leads to a definition of the indiscernibility relation \textquotedblleft$\sim_{B}%
$\textquotedblright.

Let $x,x^{\prime}\in\mathcal{O}$, $B\subseteq \mathcal{F}$.

\[
\begin{tabular}
[c]{l}%
$\sim_{B}=\left\{  \left(  x,x^{\prime}\right)  \in\mathcal{O}\times
\mathcal{O}\mid\forall\varphi_{i}\in B\text{ },\text{ }\Delta_{\varphi_{i}%
}=0\right\}  $%
\end{tabular}
\]

is called the indiscernibility relation on $\mathcal{O}$, where description
length $i\leq\left\vert \Phi\right\vert $.%

\[%
\begin{array}
[c]{c|l}%
Symbol & Interpretation\\\hline
B & B\subseteq%
%TCIMACRO{\tciFourier}%
%BeginExpansion
\mathcal{F}%
%EndExpansion
\text{,}\\
r & \binom{\left\vert B\right\vert }{r}\text{, i.e. , }\left\vert B\right\vert
\text{ probe functions }\varphi_{i}\in B\text{ taken }r\text{ at a time,}\\
B_{r} & r\leq\left\vert B\right\vert \text{ probe functions in }B\text{,}\\
\sim_{B_{r}} & \text{Indiscernibility relation defined using }B_{r}\text{,}\\
\left[  x\right]  _{B_{r}} & \left[  x\right]  _{B_{r}}=\left\{  x^{\prime}%
\in\mathcal{O}\mid x\sim_{B_{r}}x^{\prime}\right\}  \text{, equivalence
(nearness) class,}\\
\mathcal{O}\diagup\sim_{B_{r}} & \mathcal{O}\diagup\sim_{B_{r}}=\left\{
\left[  x\right]  _{B_{r}}\mid x\in\mathcal{O}\right\}  \text{, quotient
set,}\\
\xi_{\mathcal{O},B_{r}} & \text{Partition }\xi_{\mathcal{O},B_{r}}%
=\mathcal{O}\diagup\sim_{B_{r}}\text{,}\\
N_{r}\left(  B\right)  & N_{r}\left(  B\right)  =\left\{  \xi_{\mathcal{O}%
,B_{r}}\mid B_{r}\subseteq B\right\}  \text{, set of partitions,}\\
{\nu}_{N_{r}} & {\nu}_{N_{r}}:\wp\left(  \mathcal{O}\right)
\times\wp\left(  \mathcal{O}\right)  \longrightarrow\left[  0,1\right]
\text{, overlap function,}\\
N_{r}\left(  B\right)  _{\ast}X & N_{r}\left(  B\right)  _{\ast}%
X=\bigcup_{\left[  x\right]  _{B_{r}}\subseteq X}\left[  x\right]  _{B_{r}%
}\text{, lower approximation,}\\
N_{r}\left(  B\right)  ^{\ast}X & N_{r}\left(  B\right)  ^{\ast}%
X=\bigcup_{\left[  x\right]  _{B_{r}}\cap X\neq\varnothing}\left[  x\right]
_{B_{r}}\text{, upper approximation,}\\
 Bnd_{N_{r}\left(  B\right)  }\left(  X\right)  & N_{r}\left(  B\right)
^{\ast}X\diagdown N_{r}\left(  B\right)  _{\ast}X=\left\{  x\in N_{r}\left(
B\right)  ^{\ast}X\mid x\notin N_{r}\left(  B\right)  _{\ast}X\right\}.%
\end{array}
\]

\begin{center}
\bigskip Table 1 : Nearness Approximation Space Symbols
\end{center}

\normalsize
A nearness approximation space is a tuple $NAS=\left(  \mathcal{O},%
%TCIMACRO{\tciFourier}%
%BeginExpansion
\mathcal{F}%
%EndExpansion
,\sim_{B_{r}},N_{r}(B),{\nu}_{N_{r}}\right)  $ where the approximation
space $NAS$ is defined with a set of perceived objects $\mathcal{O}$, set of
probe functions $%
%TCIMACRO{\tciFourier}%
%BeginExpansion
\mathcal{F}%
%EndExpansion
$\ representing object features, indiscernibility relation $\sim_{B_{r}\text{
}}$defined relative to $B_{r}\subseteq B\subseteq%
%TCIMACRO{\tciFourier}%
%BeginExpansion
\mathcal{F}%
%EndExpansion
$, collection of partitions (families of neighbourhoods) $N_{r}\left(
B\right)  $, and overlap function ${\nu}_{N_{r}}$. The subscript $r$
denotes the cardinality of the restricted subset $B_{r}$, where we consider
$\binom{\left\vert B\right\vert }{r}$, i.e., $\left\vert B\right\vert $
functions $\phi_{i}\in%
%TCIMACRO{\tciFourier}%
%BeginExpansion
\mathcal{F}%
%EndExpansion
$ taken $r$ at a time to define the relation $\sim_{B_{r}\text{ }}$. This
relation defines a partition of $\mathcal{O}$\ into non-empty, pairwise
disjoint subsets that are equivalence classes denoted by $\left[  x\right]
_{B_{r}}$, where $\left[  x\right]  _{B_{r}}=\left\{  x^{\prime}\in
\mathcal{O}\mid x\sim_{B_{r}}x^{\prime}\right\}  $. These classes form a new
set called the quotient set $\mathcal{O}\diagup\sim_{B_{r}}$, where
$\mathcal{O}\diagup\sim_{B_{r}}=\left\{  \left[  x\right]  _{B_{r}}\mid
x\in\mathcal{O}\right\}  $. In effect, each choice of probe functions $B_{r}%
$\ defines a partition $\xi_{\mathcal{O},B_{r}}$\ on a set of objects
$\mathcal{O}$, namely, $\xi_{\mathcal{O},B_{r}}=\mathcal{O}\diagup\sim_{B_{r}%
}$. Every choice of the set $B_{r}$ leads to a new partition of $\mathcal{O}$.
Let $%
%TCIMACRO{\tciFourier}%
%BeginExpansion
\mathcal{F}%
%EndExpansion
$ denote a set of features for objects in a set $X$, where each $\phi_{i}\in%
%TCIMACRO{\tciFourier}%
%BeginExpansion
\mathcal{F}%
%EndExpansion
$ that maps $X$ to some value set $V_{\phi_{i}}$ (range of $\phi_{i}$). The
value of $\phi_{i}\left(  x\right)  $ is a measurement associated with a
feature of an object $x\in X$. The overlap function ${\nu}_{N_{r}}$ is
defined by ${\nu}_{N_{r}}:\wp\left(  \mathcal{O}\right)  \times
\wp\left(  \mathcal{O}\right)  \longrightarrow\left[  0,1\right]  $, where
$\wp\left(  \mathcal{O}\right)  $ is the powerset of $\mathcal{O}$. The
overlap function ${\nu}_{N_{r}}$ maps a pair of sets to a number in
$\left[  0,1\right]  $\ representing the degree of overlap between sets of
objects with their features defined by probe functions $B_{r}\subseteq B$
\cite{Skowron}. For each subset $B_{r}\subseteq B$ of probe functions, define
the binary relation $\sim_{B_{r}}=\left\{  \left(  x,x^{\prime}\right)
\in\mathcal{O}\times\mathcal{O}\mid\forall\phi_{i}\in B_{r},\text{ }\phi
_{i}\left(  x\right)  =\phi_{i}\left(  x^{\prime}\right)  \right\}  $. Since
each $\sim_{B_{r}}$\ is, in fact, the usual indiscernibility relation, for
$B_{r}\subseteq B$\ and $x\in\mathcal{O}$, let $\left[  x\right]  _{B_{r}}$
denote the equivalence class containing $x$. If $\left(  x,x^{\prime
}\right)  \in\sim_{B_{r}\text{ }}$, then $x$ and $x^{\prime}$ are said to be
$B$-indiscernible with respect to all feature probe functions in $B_{r}$. Then
define a collection of partitions $N_{r}\left(  B\right)  $, where
$N_{r}\left(  B\right)  =\left\{  \xi_{\mathcal{O},B_{r}}\mid B_{r}\subseteq
B\right\}  $.

\subsection{Descriptively Near Sets}

We need the notion of nearness between sets, and so we consider the concept of
the descriptively near sets. In 2007, descriptively near sets were introduced
as a means of solving classification and pattern recognition problems arising
from disjoint sets that resemble each other \cite{Peters, Peters1}.

A set of objects $A\subseteq\mathcal{O}$ is characterized by the unique
description of each object in the set.

Set Description\textbf{: }\cite{Naimpally} Let $\mathcal{O}$ be a set of
perceptual objects, $\Phi$ an object description and $A\subseteq\mathcal{O}$.
Then the \textit{set description} of $A$ is defined as%

\[
\mathcal{Q}(A)=\{\Phi(a)\mid a\in A\}.
\]

Descriptive Set Intersection\textbf{: }\cite{Naimpally, Peters6} Let
$\mathcal{O}$ be a set of perceptual objects, $A$ and $B$ any two subsets of
$\mathcal{O}$. Then the descriptive (set) intersection of $A$ and $B$ is
defined as%

\[
A\underset{\Phi}{\cap}B=\left\{  x\in A\cup B\mid\Phi\left(  x\right)
\in\mathcal{Q}\left(  A\right)  \text{ }and\text{ }\Phi\left(  x\right)
\in\mathcal{Q}\left(  B\right)  \right\}  \text{.}%
\]

If $\mathcal{Q}(A)\cap\mathcal{Q}(B)\neq\emptyset$, then $A$ is called
descriptively near $B$ and denoted by $A\delta_{\Phi}B$ \cite{Peters0}.

Descriptive Nearness Collections\textbf{:} \cite{Peters0} $\xi_{\Phi}\left(
A\right)  =\left\{  B\in\mathcal{P}\left(  \mathcal{O}\right)  \mid
A\delta_{\Phi}B\right\}  $.

Let $\Phi$ be an object description, $A$ any subset of $\mathcal{O}$ and
$\xi_{\Phi}\left(  A\right)  $ a descriptive nearness collections. Then
$A\in\xi_{\Phi}\left(  A\right)  $ \cite{Peters0}.

\subsection{Some Algebraic Structures on NAS}

A \textit{binary operation} on a set $G$ is a mapping of $G\times G$ into $G$,
where $G\times G$ is the set of all ordered pairs of elements of $G$. A
\textit{groupoid} is a system $G\left(  \cdot\right)  $ consisting of a
nonempty set $G$ together with a binary operation \textquotedblleft$\cdot
$\textquotedblright\ on $G$ \cite{Clifford1961}.

Let $NAS=\left(  \mathcal{O},%
%TCIMACRO{\tciFourier}%
%BeginExpansion
\mathcal{F}%
%EndExpansion
,\sim_{B_{r}},N_{r}\left(  B\right),{\nu}_{N_{r}}\right)  $ be a nearness approximation
space $(NAS)$ and let \textquotedblleft$\cdot$\textquotedblright\ a binary
operation defined on $\mathcal{O}$. A subset $G$ of the set of perceptual
objects $\mathcal{O}$\ is called a \textit{near group on }$NAS$ if the
following properties are satisfied:

\begin{enumerate}
\item[$(NG_{1})$] For all $x,y\in G$, $x\cdot y\in N_{r}\left(  B\right)
^{\ast}G$,

\item[$(NG_{2})$] For all $x,y,z\in G$, $\left(  x\cdot y\right)  \cdot
z=x\cdot\left(  y\cdot z\right)  $ property\ holds in $N_{r}\left(  B\right)
^{\ast}G$,

\item[$(NG_{3})$] There exists $e\in N_{r}\left(  B\right)  ^{\ast}G$ such
that $x\cdot e=e\cdot x=x$ for all $x\in G$\emph{ }($e$ is called the near
identity element of $G$),

\item[$(NG_{4})$] There exists $y\in G$ such that $x\cdot y=y\cdot x=e$ for
all $x\in G$ ($y$ is called the near inverse of $x$ in $G$ and denoted as
$x^{-1}$) \cite{Inan}.
\end{enumerate}

If in addition, for all $x,y\in G$, $x\cdot y=y\cdot x$ property holds in
$N_{r}\left(  B\right)  ^{\ast}G$, then $G$ is said to be an abelian near
group on $NAS$.

Also, a nonempty subset $S\subseteq\mathcal{O}$ is called a \textit{near
semigroup} on $NAS$ if $x\cdot y\in N_{r}\left(  B\right)  ^{\ast}S$ for all
$x,y\in S$ and $\left(  x\cdot y\right)  \cdot z=x\cdot\left(  y\cdot
z\right)  $ for all $x,y,z\in S$ property\ holds in $N_{r}\left(  B\right)
^{\ast}\left(  S\right)  $.

\begin{theorem}
\cite{Inan}\label{Th01}Let $G$ be a near group on $NAS$.

(i) There exists a unique near identity element $e\in N_{r}\left(  B\right)
^{\ast}G$ such that $x\cdot e=x=e\cdot x$ for all $x\in G$.

(ii) For all $x\in G$, there exists a unique $y\in G$ such that $x\cdot
y=e=y\cdot x$.
\end{theorem}

\begin{theorem}
\cite{Inan}\label{Th02}Let $G$ be a near group on $NAS$.

(i) $\left(  x^{-1}\right)  ^{-1}=x$ for all $x\in G$.

(ii) If $x\cdot y\in G$, then $\left(  x\cdot y\right)  ^{-1}=y^{-1}\cdot
x^{-1}$ for all $x,y\in G$.

(iii) If either $x\cdot z=y\cdot z$ or $z\cdot x=z\cdot y$, then $x=y$ for all
$x,y,z\in G$.
\end{theorem}

$H$\ is called a subnear group of near group $G$ if $H$ is a near group
relative to the operation in $G$. There is only one guaranteed trivial subnear
group of near group $G$, i.e., $G$ itself. Moreover, $\left\{  e\right\}  $ is
a trivial subnear group of near group $G$ if and only if $e\in G$.

\begin{theorem}
\cite{Inan2014}\label{Th001} Let $G$ be a near group on nearness approximation
space, $H$ be a nonempty subset of $G$ and $N_{r}\left(  B\right)  ^{\ast}H$
be a groupoid. $H\subseteq G$ is a subnear group of $G$ if and only if
$x^{-1}\in H$ for all $x\in H$.
\end{theorem}

Let $H_{1}$ and $H_{2}$ be two near subgroups of \ the near group $G$ and
$N_{r}\left(  B\right)  ^{\ast}H_{1}$, $N_{r}\left(  B\right)  ^{\ast}H_{2}$
groupoids. If $\left(  N_{r}\left(  B\right)  ^{\ast}H_{1}\right)  \cap\left(
N_{r}\left(  B\right)  ^{\ast}H_{2}\right)  =N_{r}\left(  B\right)  ^{\ast
}\left(  H_{1}\cap H_{2}\right)  $, then $H_{1}\cap H_{2}$ is a near subgroup
of near group $G$ \cite{Inan2014}.

Let $G\subset\mathcal{O}$ be a near group and $H$ be a subnear group of $G$.
The left weak equivalence relation (compatible relation) \textquotedblleft%
$\sim_{L}$\textquotedblright\ defined as%
\[
a\sim_{L}b:\Leftrightarrow a^{-1}\cdot b\in H\cup\left\{  e\right\}  \text{.}%
\]

A weak class defined by relation \textquotedblleft$\sim_{L}$\textquotedblright%
\ is called left weak coset. The left weak coset that contains the element $a$
is denoted by $\tilde{a}_{L}$, i.e.%
\[
\tilde{a}_{L}=\left\{  a\cdot h\mid h\in H,\text{ }a\in G,\text{ }a\cdot h\in
G\right\}  \cup\left\{  a\right\}  =aH\text{.}%
\]

Let $\left(  \mathcal{O}_{1},%
%TCIMACRO{\tciFourier}%
%BeginExpansion
\mathcal{F}%
%EndExpansion
_{1},\sim_{B_{r_{1}}},N_{r_{1}}\left(  B\right),{\nu}_{N_{r_{1}}}\right)  $ and
$\left(  \mathcal{O}_{2},%
%TCIMACRO{\tciFourier}%
%BeginExpansion
\mathcal{F}%
%EndExpansion
_{2},\sim_{B_{r_{2}}},N_{r_{2}}\left(  B\right),{\nu}_{N_{r_{2}}}\right)  $ be two
nearness approximation spaces and \textquotedblleft$\cdot$\textquotedblright,
\textquotedblleft$\circ$\textquotedblright\ binary operations over
$\mathcal{O}_{1}$ and $\mathcal{O}_{2}$, respectively.

Let $G_{1}\subset\mathcal{O}_{1}$, $G_{2}\subset\mathcal{O}_{2}$ be two near
groups and $\sigma$ a mapping from $N_{r_{1}}\left(  B\right)  ^{\ast}G_{1}$
onto $N_{r_{2}}\left(  B\right)  ^{\ast}G_{2}$. If $\sigma\left(  x\cdot
y\right)  =\sigma\left(  x\right)  \circ\sigma\left(  y\right)  $ for all
$x,y\in G_{1}$, then $\sigma$ is called a near homomorphism and also, $G_{1}$
is called near homomorphic to $G_{2}$.

Let $G_{1}\subset\mathcal{O}_{1}$, $G_{2}\subset\mathcal{O}_{2}$ be near
homomorphic groups, $H_{1}$ a near subgroup and $N_{r_{1}}\left(  B\right)
^{\ast}H_{1}$ a groupoid. If $\sigma\left(  N_{r_{1}}\left(  B\right)  ^{\ast
}H_{1}\right)  =N_{r_{2}}\left(  B\right)  ^{\ast}\sigma\left(  H_{1}\right)
$, then $\sigma\left(  H_{1}\right)  $ is a near subgroup of $G_{2}$
\cite{Inan2014}.

The kernel of $\sigma$\ is defined to be the set $Ker\sigma=\left\{  x\in
G_{1}\mid\sigma\left(  x\right)  =e^{\prime}\right\}  $, where $e^{\prime}$ is
the near identity element of $G_{2}$.

\begin{theorem}
\cite{Inan2014}\label{th002}Let $G_{1}\subset\mathcal{O}_{1}$,$G_{2}%
\subset\mathcal{O}_{2}$ be near groups that are near homomorphic,
$Ker\sigma=N$ be near homomorphism kernel and $N_{r}\left(  B\right)  ^{\ast
}N$ be a groupoid. Then $N$ is a near normal subgroup of $G_{1}$.
\end{theorem}

\begin{definition}
\cite{Ozturk2013}\label{Df18}Let $\mathcal{O}$ be a set of perceptual objects,
$G\subset\mathcal{O}$ a near group and $H$ a subnear group of $G$. Let
$G/_{\sim_{L}}$ be a set of all left weak cosets of $G$ by $H$, $\xi_{\Phi
}\left(  A\right)  $ a descriptive nearness collections and $A\in
\mathcal{P}\left(  \mathcal{O}\right)  $. Then%
\[
N_{r}\left(  B\right)  ^{\ast}\left(  G/_{\sim_{L}}\right)  =\bigcup
_{\xi_{\Phi}\left(  A\right)  \text{ }\underset{\Phi}{\cap}\text{ }%
G/_{\sim_{L}}\neq\emptyset}\xi_{\Phi}\left(  A\right)
\]

is called upper approximation of $G/_{\sim_{L}}$.
\end{definition}

\begin{theorem}
\cite{Ozturk2013}\label{Th0}Let $G$ be a near group, $H$ a subnear group of
$G$ and $G/_{\sim_{L}}$ a set of all left weak cosets of $G$ by $H$. If
$\left(  N_{r}\left(  B\right)  ^{\ast}G\right)  /_{\sim_{L}}\subseteq
N_{r}\left(  B\right)  ^{\ast}\left(  G/_{\sim_{L}}\right)  $, then
$G/_{\sim_{L}}$ is a near group under the operation given by $aH\odot
bH=\left(  a\cdot b\right)  H$ for all $a,b\in G$.
\end{theorem}

Let $G$ be a near group and $H$ a subnear group of $G$. The near group
$G/_{\sim_{L}}$ is called a near group of all left weak cosets of $G$ by $H$
and denoted by $G/_{w}H$ \cite{Ozturk2013}.

\section{Nearness Rings on Nearness Approximation Spaces}

\begin{definition}
\label{Df4}Let $NAS=\left(  \mathcal{O},%
%TCIMACRO{\tciFourier}%
%BeginExpansion
\mathcal{F}%
%EndExpansion
,\sim_{B_{r}},N_{r}\left(  B\right),{\nu}_{N_{r}}\right)  $ be a nearness approximation
space and \textquotedblleft$+$\textquotedblright\ and \textquotedblleft$\cdot
$\textquotedblright\ binary operations defined on $\mathcal{O}$. A subset $R$
of the set of perceptual objects $\mathcal{O}$\ is called a nearness ring on
$NAS$ if the following properties are satisfied:

\begin{enumerate}
\item[$(NR_{1})$] $R$ is an abelian near group on $NAS$ with binary operation
\textquotedblleft$+$\textquotedblright,

\item[$(NR_{2})$] $R$ is a near semigroup on $NAS$ with binary operation
\textquotedblleft$\cdot$\textquotedblright,

\item[$(NR_{3})$] For all $x,y,z\in R$,

$x\cdot\left(  y+z\right)  =\left(  x\cdot y\right)  +\left(  x\cdot z\right)
$ and

$\left(  x+y\right)  \cdot z=\left(  x\cdot z\right)  +\left(  y\cdot
z\right)  $ properties\ hold in $N_{r}\left(  B\right)  ^{\ast}R$.
\end{enumerate}

If in addition:

\begin{enumerate}
\item[$(NR_{4})$] $x\cdot y=y\cdot x$ for all $x,y\in R$,
\end{enumerate}

then $R$ is said to be a commutative\textit{ nearness ring}.

\begin{enumerate}
\item[$(NR_{5})$] If $N_{r}\left(  B\right)  ^{\ast}R$ contains an element
$1_{R}$ such that $1_{R}\cdot x=x\cdot1_{R}=x$ for all $x\in R$,
\end{enumerate}

then $R$ is said to be a \textit{nearness ring with identity}.
\end{definition}

These properties have to hold in $N_{r}\left(  B\right)  ^{\ast}R$. Sometimes
they may be hold in $\mathcal{O\diagup}N_{r}\left(  B\right)  ^{\ast}R$, then
$R$ is not a nearness ring on $NAS$. Multiplying or sum of finite number of
elements in $R$ may not always belongs to $N_{r}\left(  B\right)  ^{\ast}R$.
Therefore always we can not say that $x^{n}\in N_{r}\left(  B\right)  ^{\ast
}R$ or $nx\in N_{r}\left(  B\right)  ^{\ast}R$, for all $x\in R$ and some
positive integer $n$. If $\left(  N_{r}\left(  B\right)  ^{\ast}R,+\right)  $
and $\left(  N_{r}\left(  B\right)  ^{\ast}R,\cdot\right)  $ are groupoids,
then we can say that $x^{n}\in N_{r}\left(  B\right)  ^{\ast}R$ for all
positive integer $n$\ or $nx\in N_{r}\left(  B\right)  ^{\ast}R$ all integer
$n$, for all $x\in R$.

An element $x$ in nearness ring $R$ with identity is said to be \textit{left}
(resp. \textit{right}) \textit{invertible} if there exists $y\in N_{r}\left(
B\right)  ^{\ast}R$ (resp. $z\in N_{r}\left(  B\right)  ^{\ast}R$) such that
$y\cdot x=1_{R}$ (resp. $x\cdot z=1_{R}$). The element $y$ (resp. $z$) is
called a \textit{left} (resp. \textit{right}) \textit{inverse} of $x$. If
$x\in R$ is both left and right invertible, then $x$ is said to be
\textit{nearness invertible} or \textit{nearness unit}. The set of nearness
units in a nearness ring $R$ with identity forms is a near group on $NAS$ with multiplication.

A nearness ring $R$ is a \textit{nearness division ring} iff $\left(
R\backslash\left\{  0\right\}  ,\cdot\right)  $ is a near group on $NAS$,
i.e., every nonzero elements in $R$ is a nearness unit. Similarly, a nearness
ring $R$ is a \textit{nearness field} iff $\left(  R\backslash\left\{
0\right\}  ,\cdot\right)  $ is a commutative near group on $NAS$.

Some elementary properties of elements in nearness rings are not always
provided as in ordinary rings. If we consider $N_{r}\left(  B\right)  ^{\ast
}R$\ as a ordinary ring, then elementary properties of elements in nearness
ring are provided.

\begin{lemma}
Every ordinary rings on $NAS$ are nearness rings on $NAS$.
\end{lemma}

\begin{example}
\label{Ex1}Let $\mathcal{O}=\left\{  o,p,r,s,t,v,w,x\right\}  $ be a set of
perceptual objects and $B=\left\{  \varphi_{1},\varphi_{2},\varphi
_{3}\right\}  \subseteq%
%TCIMACRO{\tciFourier}%
%BeginExpansion
\mathcal{F}%
%EndExpansion
$ a set of probe functions. Values of the probe functions%
\begin{align*}
\varphi_{1}  &  :\mathcal{O}\longrightarrow V_{1}=\left\{  \alpha_{1}%
,\alpha_{2},\alpha_{3},\alpha_{4}\right\}  \text{,}\\
\varphi_{2}  &  :\mathcal{O}\longrightarrow V_{2}=\left\{  \beta_{1},\beta
_{2},\beta_{3}\right\}  \text{,}%
\end{align*}
are given in Table 2.%

\[%
\begin{array}
[c]{c|cccccccc}
& o & p & r & s & t & v & w & x\\\hline
\varphi_{1} & \alpha_{4} & \alpha_{2} & \alpha_{1} & \alpha_{2} & \alpha_{1} &
\alpha_{3} & \alpha_{4} & \alpha_{3}\\
\varphi_{2} & \beta_{1} & \beta_{3} & \beta_{2} & \beta_{3} & \beta_{2} &
\beta_{3} & \beta_{1} & \beta_{3}%
\end{array}
\]

\[
Table\text{ }2.
\]

Let \textquotedblleft$+$\textquotedblright\ and \textquotedblleft$\cdot
$\textquotedblright\ be binary operations of perceptual objects on
$\mathcal{O}$ as in Tables 3 and 4.%

\[%
\begin{array}
[c]{c|cccccccc}%
+ & o & p & r & s & t & v & w & x\\\hline
o & \multicolumn{1}{c}{o} & p & r & s & t & v & w & x\\
p & \multicolumn{1}{c}{p} & r & s & t & v & w & x & o\\
r & \multicolumn{1}{c}{r} & s & t & v & w & x & o & p\\
s & \multicolumn{1}{c}{s} & t & v & w & x & o & p & r\\
t & \multicolumn{1}{c}{t} & v & w & x & o & p & r & s\\
v & \multicolumn{1}{c}{v} & w & x & p & p & r & s & t\\
w & \multicolumn{1}{c}{w} & x & o & p & r & s & t & v\\
x & \multicolumn{1}{c}{x} & o & p & r & s & t & v & w
\end{array}
\text{ \ \ \ \ \ \ \ }%
\begin{array}
[c]{c|cccccccc}%
\cdot & o & p & r & s & t & v & w & x\\\hline
o & \multicolumn{1}{c}{o} & o & o & o & o & o & o & o\\
p & \multicolumn{1}{c}{o} & p & r & s & t & v & w & x\\
r & \multicolumn{1}{c}{o} & r & t & w & o & r & t & w\\
s & \multicolumn{1}{c}{o} & s & w & p & t & o & r & v\\
t & \multicolumn{1}{c}{o} & t & o & t & o & t & o & t\\
v & \multicolumn{1}{c}{o} & v & r & x & t & p & w & s\\
w & \multicolumn{1}{c}{o} & w & t & r & o & w & t & r\\
x & \multicolumn{1}{c}{o} & x & w & v & t & s & r & p
\end{array}
\]

\[
Table\text{ }3.\text{
\ \ \ \ \ \ \ \ \ \ \ \ \ \ \ \ \ \ \ \ \ \ \ \ \ \ \ \ \ \ \ \ \ \ \ \ \ }%
Table\text{ }4.
\]

Since $r+\left(  s+s\right)  \neq\left(  r+s\right)  +s$, $\left(
\mathcal{O},\mathcal{+}\right)  $ is not a group, i.e., $\left(
\mathcal{O},\mathcal{+},\cdot\right)  $ is not a ring. Let $R=\left\{
r,t,w\right\}  $ be a subset of perceptual objects. Let \textquotedblleft%
$+$\textquotedblright\ and \textquotedblleft$\cdot$\textquotedblright\ be
operations of perceptual objects on $R\subseteq\mathcal{O}$ as in Tables 5 and 6.%

\[%
\begin{array}
[c]{c|ccc}%
+ & r & t & w\\\hline
r & t & w & o\\
t & w & o & r\\
w & o & r & t
\end{array}
\text{ \ \ \ \ \ \ \ \ \ \ }%
\begin{array}
[c]{c|ccc}%
\cdot & r & t & w\\\hline
r & t & o & t\\
t & o & o & o\\
w & t & o & t
\end{array}
\]%
\[
\text{\ }Table\text{ }5.\text{ \ \ \ \ \ \ \ \ \ \ \ \ \ \ \ \ \ \ \ }%
Table\text{ }6.
\]

\[
\left[  o\right]  _{\varphi_{1}}=\left\{  x^{\prime}\in\mathcal{O}\mid
\varphi_{1}\left(  x^{\prime}\right)  =\varphi_{1}\left(  o\right)
=\alpha_{4}\right\}  =\left\{  o,w\right\}  \text{,}%
\]%
\begin{align*}
\left[  p\right]  _{\varphi_{1}}  &  =\left\{  x^{\prime}\in\mathcal{O}%
\mid\varphi_{1}\left(  x^{\prime}\right)  =\varphi_{1}\left(  p\right)
=\alpha_{2}\right\} \\
&  =\left\{  p,s\right\} \\
&  =\left[  s\right]  _{\varphi_{1}}\text{,}%
\end{align*}%
\begin{align*}
\left[  r\right]  _{\varphi_{1}}  &  =\left\{  x^{\prime}\in\mathcal{O}%
\mid\varphi_{1}\left(  x^{\prime}\right)  =\varphi_{1}\left(  r\right)
=\alpha_{1}\right\} \\
&  =\left\{  r,t\right\} \\
&  =\left[  t\right]  _{\varphi_{1}}\text{,}%
\end{align*}%
\begin{align*}
\left[  v\right]  _{\varphi_{1}}  &  =\left\{  x^{\prime}\in\mathcal{O}%
\mid\varphi_{1}\left(  x^{\prime}\right)  =\varphi_{1}\left(  v\right)
=\alpha_{3}\right\} \\
&  =\left\{  v,x\right\} \\
&  =\left[  x\right]  _{\varphi_{1}}\text{.}%
\end{align*}

Hence we have that $\xi_{\varphi_{1}}=\left\{  \left[  o\right]  _{\varphi
_{1}},\left[  r\right]  _{\varphi_{1}},\left[  v\right]  _{\varphi_{1}%
},\left[  w\right]  _{\varphi_{1}}\right\}  $.%
\begin{align*}
\left[  o\right]  _{\varphi_{2}}  &  =\left\{  x^{\prime}\in\mathcal{O}%
\mid\varphi_{2}\left(  x^{\prime}\right)  =\varphi_{2}\left(  o\right)
=\beta_{1}\right\} \\
&  =\left\{  o,w\right\} \\
&  =\left[  w\right]  _{\varphi_{2}}\text{,}%
\end{align*}%
\begin{align*}
\left[  p\right]  _{\varphi_{2}}  &  =\left\{  x^{\prime}\in\mathcal{O}%
\mid\varphi_{2}\left(  x^{\prime}\right)  =\varphi_{2}\left(  p\right)
=\beta_{3}\right\} \\
&  =\left\{  p,s,v,x\right\} \\
&  =\left[  s\right]  _{\varphi_{2}}=\left[  v\right]  _{\varphi_{2}}=\left[
x\right]  _{\varphi_{2}}\text{,}%
\end{align*}%
\begin{align*}
\left[  r\right]  _{\varphi_{2}}  &  =\left\{  x^{\prime}\in\mathcal{O}%
\mid\varphi_{2}\left(  x^{\prime}\right)  =\varphi_{2}\left(  r\right)
=\beta_{2}\right\} \\
&  =\left\{  r,t\right\} \\
&  =\left[  t\right]  _{\varphi_{2}}\text{.}%
\end{align*}

Thus we obtain that $\xi_{\varphi_{2}}=\left\{  \left[  o\right]
_{\varphi_{2}},\left[  p\right]  _{\varphi_{2}},\left[  r\right]
_{\varphi_{2}}\right\}  $.

Therefore, for $r=1$, a set of partitions of $\mathcal{O}$\ is $N_{1}\left(
B\right)  =\left\{  \xi_{\varphi_{1}},\xi_{\varphi_{2}}\right\}  $.

In this case, we can write%
\begin{align*}
N_{1}\left(  B\right)  ^{\ast}R  &  =%
%TCIMACRO{\QTATOP{\bigcup\text{ }\left[  x\right]  _{\varphi_{i}}}{\left[
%x\right]  _{\varphi_{i}}\text{ }\cap\text{ }R\neq\varnothing}}%
%BeginExpansion
\genfrac{}{}{0pt}{1}{\bigcup\text{ }\left[  x\right]  _{\varphi_{i}}}{\left[
x\right]  _{\varphi_{i}}\text{ }\cap\text{ }R\neq\varnothing}%
%EndExpansion
\\
&  =\left\{  r,t\right\}  \cup\left\{  o,w\right\}  \cup\left\{  o,w\right\}
\cup\left\{  r,t\right\} \\
&  =\left\{  o,r,t,w\right\}  \neq\mathcal{O}\text{.}%
\end{align*}

From Definition \ref{Df4}, since

\begin{enumerate}
\item[$(NR_{1})$] $R$ is an abelian near group on $NAS$ with binary operation
\textquotedblleft$+$\textquotedblright,

\item[$(NR_{2})$] $R$ is a near semigroup on $NAS$ with binary operation
\textquotedblleft$\cdot$\textquotedblright\ and

\item[$(NR_{3})$] For all $x,y,z\in R$,

$x\cdot\left(  y+z\right)  =\left(  x\cdot y\right)  +\left(  x\cdot z\right)
$ and

$\left(  x+y\right)  \cdot z=\left(  x\cdot z\right)  +\left(  y\cdot
z\right)  $ properties\ hold in $N_{r}\left(  B\right)  ^{\ast}R$.
\end{enumerate}

conditions hold, $R$ is a nearness ring on $NAS$.
\end{example}

\begin{proposition}
Let $R$ be a nearness ring on $NAS$ and $0\in R$. If $0\cdot x,x\cdot0\in R$,
then for all $x,y\in R$

(i) $x\cdot0=0\cdot x=0$,

(ii) $x\cdot\left(  -y\right)  =\left(  -x\right)  \cdot y=-\left(  x\cdot
y\right)  $,

(iii) $\left(  -x\right)  \cdot\left(  -y\right)  =x\cdot y$.
\end{proposition}

\begin{definition}
\label{Df5}Let $R$ be a nearness ring on $NAS$ and $S$ a nonempty subset of
$R$. $S$ is called subnearness ring of $R$, if $S$ is a nearness ring with
binary operations \textquotedblleft$+$\textquotedblright\ and
\textquotedblleft$\cdot$\textquotedblright\ on nearness ring $R$.
\end{definition}

\begin{definition}
Let we consider nearness field $R$ and a nonempty subset $S$ of $R$. $S$ is
called subnearness field of $R$, if $S$ is a nearness field.
\end{definition}

\begin{theorem}
\label{Th1}Let $R$ be a nearness ring on $NAS$ and $(N_{r}\left(  B\right)
^{\ast}S,+)$, $(N_{r}\left(  B\right)  ^{\ast}S,\cdot)$ groupoids. A nonempty
subset $S$ of nearness ring $R$ is a subnearness ring of $R$ iff $-x\in S$ for
all $x\in S$.
\end{theorem}

\begin{proof}
Suppose that $S$ is a subnearness ring of $R$. Then $S$ is a nearness ring and
$-x\in S$ for all $x\in S$. Conversely, suppose $-x\in S$ for all $x\in S$.
Then since $(N_{r}\left(  B\right)  ^{\ast}S,+)$ is a groupoid, from Theorem
\ref{Th001} $\left(  S,+\right)  $ is a commutative near group on $NAS$. By
the hypothesis, since $(N_{r}\left(  B\right)  ^{\ast}S,\cdot)$ is a groupoid
and $S\subseteq R$, then associative property holds in $N_{r}\left(  B\right)
^{\ast}S$. Hence $\left(  S,\cdot\right)  $ is a near semigroup on $NAS$. For
all $x,y,z\in S\subseteq R$, $y+z\in N_{r}\left(  B\right)  ^{\ast}S$ and
$x\cdot(y+z)\in N_{r}\left(  B\right)  ^{\ast}S$. Also $x\cdot y+x\cdot z\in
N_{r}\left(  B\right)  ^{\ast}S$. Since $R$ is a nearness ring, $x\cdot\left(
y+z\right)  =\left(  x\cdot y\right)  +\left(  x\cdot z\right)  $ property
holds in $N_{r}\left(  B\right)  ^{\ast}S$. Similarly we can show that
$\left(  x+y\right)  \cdot z=\left(  x\cdot z\right)  +\left(  y\cdot
z\right)  $ property holds in $N_{r}\left(  B\right)  ^{\ast}S$. Therefore $S$
is a subnearness ring of nearness ring $R$.
\end{proof}

\begin{example}
\label{Ex2}From Example \ref{Ex1}, let we consider the nearness ring
$R=\left\{  r,t,w\right\}  $ on $NAS$. Let $S=\left\{  r,w\right\}  \ $be a
subset of nearness ring $R$. Then, \textquotedblleft$+$\textquotedblright\ and
\textquotedblleft$\cdot$\textquotedblright\ are binary operations of
perceptual objects on $S\subseteq R$ as in Tables 7 and 8.%

\[%
\begin{array}
[c]{c|cc}%
+ & r & w\\\hline
r & t & o\\
w & o & t
\end{array}
\text{ \ \ \ \ \ \ \ \ \ \ }%
\begin{array}
[c]{c|cc}%
\cdot & r & w\\\hline
r & t & t\\
w & t & t
\end{array}
\]%
\[
\text{\ }Table\text{ }7.\text{ \ \ \ \ \ \ \ \ \ \ \ \ \ }Table\text{ }8.
\]

We know from Example \ref{Ex1}, for $r=1$, a classification of $\mathcal{O}%
$\ is $N_{1}\left(  B\right)  =\left\{  \xi_{\left(  \varphi_{1}\right)  }%
,\xi_{\left(  \varphi_{2}\right)  }\right\}  $.

Then, we can obtain $N_{1}\left(  B\right)  ^{\ast}S=\left\{  o,r,t,w\right\}
$.

Hence we can observe that $(N_{r}\left(  B\right)  ^{\ast}S,+)$,
$(N_{r}\left(  B\right)  ^{\ast}S,\cdot)$ are groupoids and $-r=w,-w=r\in
N_{r}\left(  B\right)  ^{\ast}S$. Therefore from Theorem \ref{Th1}, $S$ is a
subnearness ring of nearness ring $R$.
\end{example}

\begin{theorem}
\label{Th7}Let $R$ be a nearness ring on $NAS$, $S_{1}$ and $S_{2}$ two
subnearness rings of $R$ and $N_{r}\left(  B\right)  ^{\ast}S_{1}$,
$N_{r}\left(  B\right)  ^{\ast}S_{2}$ groupoids with the binary operations
\textquotedblleft$+$\textquotedblright\ and \textquotedblleft$\cdot
$\textquotedblright. If%
\[
\left(  N_{r}\left(  B\right)  ^{\ast}S_{1}\right)  \cap\left(  N_{r}\left(
B\right)  ^{\ast}S_{2}\right)  =N_{r}\left(  B\right)  ^{\ast}\left(
S_{1}\cap S_{2}\right)  \text{,}%
\]
then $S_{1}\cap S_{2}$ is a subnearness ring of $R$.
\end{theorem}

\begin{corollary}
Let $R$ be a nearness ring on $NAS$, $\left\{  S_{i}:i\in\Delta\right\}  $ a
nonempty family of subnearness rings of $R$ and $N_{r}\left(  B\right)
^{\ast}S_{i}$ groupoids. If%
\[%
%TCIMACRO{\tbigcap \limits_{i\in\Delta}}%
%BeginExpansion
{\textstyle\bigcap\limits_{i\in\Delta}}
%EndExpansion
\left(  N_{r}\left(  B\right)  ^{\ast}S_{i}\right)  =N_{r}\left(  B\right)
^{\ast}\left(
%TCIMACRO{\tbigcap \limits_{i\in\Delta}}%
%BeginExpansion
{\textstyle\bigcap\limits_{i\in\Delta}}
%EndExpansion
S_{i}\right)  \text{,}%
\]

then $%
%TCIMACRO{\tbigcap \limits_{i\in\Delta}}%
%BeginExpansion
{\textstyle\bigcap\limits_{i\in\Delta}}
%EndExpansion
S_{i}$ is a subnearness ring of $R$.
\end{corollary}

\begin{definition}
\label{Df6}Let $R$ be a nearness ring on $NAS$ and $I$ be a nonempty subset of
$R$. $I$ is a left (right) nearness ideal of $R$ provided for all $x,y\in I$
and for all $r\in R$, $x-y\in N_{r}\left(  B\right)  ^{\ast}I$, $r\cdot x\in
N_{r}\left(  B\right)  ^{\ast}I$ ($x-y\in N_{r}\left(  B\right)  ^{\ast}I$,
$x\cdot r\in N_{r}\left(  B\right)  ^{\ast}I$).
\end{definition}

A nonempty set $I$ of a nearness ring $R$ is called a nearness ideal of $R$ if
$I$ is both a left and a right nearness ideal of $R$.

There is only one guaranteed trivial nearness ideal of nearness ring $R$,
i.e., $R$ itself. Furthermore, $\left\{  0\right\}  $ is a trivial nearness
ideal of nearness ring $R$ iff $0\in R$.

\begin{lemma}
Every nearness ideal is a subnearness ring of nearness ring $R$.
\end{lemma}

\begin{example}
\label{Ex3}From Example \ref{Ex1} and \ref{Ex2}, let we consider the nearness
ring $R=\left\{  r,t,w\right\}  $ on $NAS$ and subnearness ring $S=\left\{
r,w\right\}  $ of $R$. We can observe that $x-y\in N_{r}\left(  B\right)
^{\ast}S$, $r\cdot x\in N_{r}\left(  B\right)  ^{\ast}S$ and $x\cdot r\in
N_{r}\left(  B\right)  ^{\ast}S$ for all $x,y\in S$ and for all $r\in R$.
Hence, from Definition \ref{Df6}, $S$ is a nearness ideal of $R$.
\end{example}

\begin{theorem}
Let $R$ be a nearness ring on $NAS$, $I_{1}$ and $I_{2}$ two nearness ideals
of $R$ and $N_{r}\left(  B\right)  ^{\ast}I_{1}$, $N_{r}\left(  B\right)
^{\ast}I_{2}$ groupoids with the binary operations \textquotedblleft%
$+$\textquotedblright\ and \textquotedblleft$\cdot$\textquotedblright. If%
\[
\left(  N_{r}\left(  B\right)  ^{\ast}I_{1}\right)  \cap\left(  N_{r}\left(
B\right)  ^{\ast}I_{2}\right)  =N_{r}\left(  B\right)  ^{\ast}\left(
I_{1}\cap I_{2}\right)  \text{,}%
\]
then $I_{1}\cap I_{2}$ is a nearness ideal of $R$.
\end{theorem}

\begin{proof}
Suppose $I_{1}$ and $I_{2}$ be two nearness ideals of the nearness ring $R$.
It is obvious that $I_{1}\cap I_{2}\subset R$. Consider $x,y\in I_{1}\cap
I_{2}$. Since $I_{1}$ and $I_{2}$ are nearness ideals, we have $x-y,r\cdot
x\in N_{r}\left(  B\right)  ^{\ast}I_{1}$ and $x-y,r\cdot x\in N_{r}\left(
B\right)  ^{\ast}I_{2}$, i.e., $x-y,r\cdot x\in\left(  N_{r}\left(  B\right)
^{\ast}I_{1}\right)  \cap\left(  N_{r}\left(  B\right)  ^{\ast}I_{2}\right)  $
for all $x,y\in I_{1},I_{2}$ and $r\in R$. Assuming $\left(  N_{r}\left(
B\right)  ^{\ast}I_{1}\right)  \cap\left(  N_{r}\left(  B\right)  ^{\ast}%
I_{2}\right)  =N_{r}\left(  B\right)  ^{\ast}\left(  I_{1}\cap I_{2}\right)
$, we have $x-y,r\cdot x\in N_{r}\left(  B\right)  ^{\ast}\left(  I_{1}\cap
I_{2}\right)  $. From Definition \ref{Df6}, $I_{1}\cap I_{2}$ is a nearness
ideal of $R$.
\end{proof}

\begin{corollary}
Let $R$ be a nearness ring on $NAS$, $\left\{  I_{i}:i\in\Delta\right\}  $ a
nonempty family of nearness ideals of $R$ and $N_{r}\left(  B\right)  ^{\ast
}I_{i}$ groupoids with the binary operations \textquotedblleft$+$%
\textquotedblright\ and \textquotedblleft$\cdot$\textquotedblright. If%
\[%
%TCIMACRO{\tbigcap \limits_{i\in\Delta}}%
%BeginExpansion
{\textstyle\bigcap\limits_{i\in\Delta}}
%EndExpansion
\left(  N_{r}\left(  B\right)  ^{\ast}I_{i}\right)  =N_{r}\left(  B\right)
^{\ast}\left(
%TCIMACRO{\tbigcap \limits_{i\in\Delta}}%
%BeginExpansion
{\textstyle\bigcap\limits_{i\in\Delta}}
%EndExpansion
I_{i}\right)  \text{,}%
\]
then $%
%TCIMACRO{\tbigcap \limits_{i\in\Delta}}%
%BeginExpansion
{\textstyle\bigcap\limits_{i\in\Delta}}
%EndExpansion
I_{i}$ is a nearness ideal of $R$.
\end{corollary}

\bigskip Let $R$ be a nearness ring and $S$ a subnearness ring of $R$. The
left weak equivalence relation (compatible relation) \textquotedblleft%
$\sim_{L}$\textquotedblright\ defined as%
\[
x\sim_{L}y:\Leftrightarrow-x+y\in S\cup\left\{  e\right\}  \text{.}%
\]

A weak class defined by relation \textquotedblleft$\sim_{L}$\textquotedblright%
\ is called left weak coset. The left weak coset that contains the element
$x\in R$ is denoted by $\tilde{x}_{L}$, i.e.,%
\[
\tilde{x}_{L}=\left\{  x+s\mid s\in S,\text{ }x\in R,\text{ }x+s\in R\right\}
\cup\left\{  x\right\}  \text{.}%
\]

Similarly we can define the right weak coset that contains the element $x\in
R$ is denoted by $\tilde{x}_{R}$, i.e.,%
\[
\tilde{x}_{R}=\left\{  s+x\mid s\in S,\text{ }x\in R,\text{ }s+x\in R\right\}
\cup\left\{  x\right\}  \text{.}%
\]
We can easily show that $\tilde{x}_{L}=x+S$ and $\tilde{x}_{R}=S+x$. Since
$(R,+)$ is a abelian near group on $NAS$, $\tilde{x}_{L}=\tilde{x}_{R}$ and so
we use only notation $\tilde{x}$. Then%
\[
R/_{\sim}=\left\{  x+S\mid x\in R\right\}
\]
is a set of all weak cosets of $R$ by $S$. In this case, if we consider
$N_{r}\left(  B\right)  ^{\ast}R$ instead of nearness ring $R$%
\[
\left(  N_{r}\left(  B\right)  ^{\ast}R\right)  /_{\sim}=\left\{  x+S\mid x\in
N_{r}\left(  B\right)  ^{\ast}R\right\}  \text{.}%
\]

\begin{definition}
\label{Df7}\cite{Ozturk2013}Let $R$ be a nearness ring and $S$ be a
subnearness ring of $R$. For $x,y\in R$, let $x+S$ and $y+S$ be two weak
cosets that determined the elements $x$ and $y$, respectively. Then sum of two
weak cosets that determined by $x+y\in N_{r}\left(  B\right)  ^{\ast}R$ can be
defined as%
\small{
\[
\left(  x+y\right)  +S=\left\{  \left(  x+y\right)  +s\mid s\in S,\text{
}x+y\in N_{r}\left(  B\right)  ^{\ast}R,\text{ }\left(  x+y\right)  +s\in
R\right\}  \cup\left\{  x+y\right\}
\]
}
\normalfont

and denoted by%
\[
\left(  x+S\right)  \oplus\left(  y+S\right)  =\left(  x+y\right)  +S\text{.}%
\]

\end{definition}

\begin{definition}
\label{Df8}Let $R$ be a nearness ring and $S$ be a subnearness ring of $R$.
For $x,y\in R$, let $x+S$ and $y+S$ be two weak cosets that determined the
elements $x$ and $y$, respectively. Then product of two weak cosets that
determined by $x\cdot y\in N_{r}\left(  B\right)  ^{\ast}R$ can be defined as%
\small{
\[
\left(  x\cdot y\right)  +S=\left\{  \left(  x\cdot y\right)  +s\mid s\in
S,\text{ }x\cdot y\in N_{r}\left(  B\right)  ^{\ast}R,\text{ }\left(  x\cdot
y\right)  +s\in R\right\}  \cup\left\{  x\cdot y\right\}
\]
}
\normalfont
and denoted by%
\[
\left(  x+S\right)  \odot\left(  y+S\right)  =\left(  x\cdot y\right)
+S\text{.}%
\]

\end{definition}

\begin{definition}
\label{Df9}Let $R/_{\sim}$ be a set of all weak cosets of $R$ by $S$,
$\xi_{\Phi}\left(  A\right)  $ a descriptive nearness collections and
$A\in\mathcal{P}\left(  \mathcal{O}\right)  $. Then%
\[
N_{r}\left(  B\right)  ^{\ast}\left(  R/_{\sim}\right)  =\bigcup_{\xi_{\Phi
}\left(  A\right)  \text{ }\underset{\Phi}{\cap}\text{ }R/_{\sim}\neq
\emptyset}\xi_{\Phi}\left(  A\right)
\]
is called upper approximation of $R/_{\sim}$.
\end{definition}

\begin{theorem}
\label{Th5}Let $R$ be a nearness ring, $S$ a subnearness ring of $R$ and
$R/_{\sim}$ be a set of all weak cosets of $R$ by $S$. If $\left(
N_{r}\left(  B\right)  ^{\ast}R\right)  /_{\sim}\subseteq N_{r}\left(
B\right)  ^{\ast}\left(  R/_{\sim}\right)  $, then $R/_{\sim}$ is a nearness
ring under the operations given by $\left(  x+S\right)  \oplus\left(
y+S\right)  =\left(  x+y\right)  +S$ and $\left(  x+S\right)  \odot\left(
y+S\right)  =\left(  x\cdot y\right)  +S$ for all $x,y\in R$.
\end{theorem}

\begin{proof}
$(NR1)$ Let $\left(  N_{r}\left(  B\right)  ^{\ast}R\right)  /_{\sim}\subseteq
N_{r}\left(  B\right)  ^{\ast}\left(  R/_{\sim}\right)  $. Since $R$ is a
nearness ring from Theorem \ref{Th0}, $\left(  R/_{\sim},\oplus\right)  $ is a
abelian near group of all weak cosets of $R$ by $S$.

$(NR2)$ Since $\left(  R,\cdot\right)  $ is a near semigroup;

$\qquad(NS1)$ We have that $x\cdot y\in N_{r}\left(  B\right)  ^{\ast}R$ and
$\left(  x+S\right)  \odot\left(  y+S\right)  =\left(  x\cdot y\right)  +S$
$\in\left(  N_{r}\left(  B\right)  ^{\ast}R\right)  /_{\sim}$ for all $\left(
x+S\right)  ,\left(  y+S\right)  \in R/_{\sim}$. From the \qquad hypothesis,
$\left(  x+S\right)  \odot \left(  y+S\right)  =\left(  x\cdot y\right)  +S$
$\in N_{r}\left(  B\right)  ^{\ast}\left(  R/_{\sim}\right)  $ for all
$\left(  x+S\right)  ,\left(  y+S\right)  \in R/_{\sim}$.

$\qquad(NS2)$ For all $x,y,z\in R/_{\sim}$, associative property hols in
$N_{r}\left(  B\right)  ^{\ast}R$. Hence for all $\left(  x+S\right)  ,\left(
y+S\right)  ,\left(  z+S\right)  \in R/_{\sim}$%

\begin{tabular}
[c]{ll}
& $\left(  \left(  x+S\right)  \odot\left(  y+S\right)  \right)  \odot\left(
z+S\right)  $\\
$=$ & $\left(  \left(  x\cdot y\right)  +S\right)  \odot\left(  z+S\right)
=\left(  \left(  x\cdot y\right)  \cdot z\right)  +S$\\
$=$ & $\left(  x\cdot\left(  y\cdot z\right)  \right)  +S=\left(  x+S\right)
\odot\left(  \left(  y\cdot z\right)  +S\right)  $\\
$=$ & $\left(  x+S\right)  \odot\left(  \left(  y+S\right)  \odot\left(
z+S\right)  \right)  $%
\end{tabular}

holds in $\left(  N_{r}\left(  B\right)  ^{\ast}R\right)  /_{\sim}$. From the
hypothesis, for all $\left(  x+S\right)$, $\left(  y+S\right)$, $\left(
z+S\right)  \in R/_{\sim}$, associative property holds in $N_{r}\left(
B\right)  ^{\ast}\left(  R/_{\sim}\right)  $. Therefore $\left(  R/_{\sim
},\odot\right)  $ is a near semigroup of all left weak cosets of $R$ by $S$.

$(NR3)$ Since $R$ is a nearness ring, left distributive law holds in
$N_{r}\left(  B\right)  ^{\ast}R$. For all $\left(  x+S\right)  ,\left(
y+S\right)  ,\left(  z+S\right)  \in R/_{\sim}$,%

\begin{tabular}
[c]{ll}
& $\left(  x+S\right)  \odot\left(  \left(  y+S\right)  \oplus\left(
z+S\right)  \right)  $\\
$=$ & $\left(  x+S\right)  \odot\left(  \left(  y+z\right)  +S\right)  $\\
$=$ & $\left(  x\cdot\left(  y+z\right)  \right)  +S=\left(  \left(  x\cdot
y\right)  +\left(  x\cdot z\right)  \right)  +S$\\
$=$ & $\left(  \left(  x\cdot y\right)  +S\right)  \oplus\left(  \left(
x\cdot z\right)  +S\right)  $\\
$=$ & $\left(  \left(  x+S\right)  \odot\left(  y+S\right)  \right)
\oplus\left(  \left(  x+S\right)  \odot\left(  z+S\right)  \right)  $.
\end{tabular}

Hence left distributive law holds in $\left(  N_{r}\left(  B\right)  ^{\ast
}R\right)  /_{\sim}$. Similarly we can show that right distributive law holds
in $\left(  N_{r}\left(  B\right)  ^{\ast}R\right)  /_{\sim}$,

$\left(  \left(  x+S\right)  \oplus\left(  y+S\right)  \right)  \odot\left(
z+S\right)  =\left(  \left(  x+S\right)  \odot\left(  z+S\right)  \right)
\oplus\left(  \left(  x+S\right)  \odot\left(  z+S\right)  \right)  $ for all
$\left(  x+S\right)  ,\left(  y+S\right)  ,\left(  z+S\right)  \in R/_{\sim}$.

From the hypothesis, distributive laws hold in $N_{r}\left(  B\right)  ^{\ast
}\left(  R/_{\sim}\right)  $. Consequently, $R/_{\sim}$ is a nearness ring.
\end{proof}

\begin{definition}
Let $R$ be a nearness ring and $S$ be a subnearness ring of $R$. The nearness
ring $R/_{\sim}$ is called a nearness ring of all weak cosets of $R$ by $S$
and denoted by $R/_{w}S$.
\end{definition}

\begin{example}
\label{Ex4}Let $S=\left\{  r,w\right\}  $ be a subset of $R=\left\{
r,t,w\right\}  $. From Example \ref{Ex2}, $S$ is a subnearness ring of
nearness ring $R$.

Now, we can compute the all weak cosets of $R$ by $S$. By using the definition
of weak coset,%
\begin{align*}
\text{ }r+S  & =\left\{  r\right\}  \cup\left\{  r\right\}  =\left\{
r\right\}  \text{, }t+S=\left\{  w,r\right\}  \cup\left\{  t\right\}
=\left\{  w,r,t\right\}  \text{, }\\
w+S  & =\left\{  t\right\}  \cup\left\{  w\right\}  =\left\{  t,w\right\}
\text{.}%
\end{align*}

Thus we have that $R/_{w}S=\left\{  r+S,t+S,w+S\right\}  $.

Since $N_{1}\left(  B\right)  ^{\ast}R=\left\{  o,r,t,w\right\}  $, we can
write the all weak cosets of $N_{1}\left(  B\right)  ^{\ast}R$ by $S$. In this
case%

\begin{center}
$o+S=\left\{  r,w\right\}  \cup\left\{  o\right\}  =\left\{  r,w,o\right\}.$
\end{center}

Then $\left(  N_{1}\left(  B\right)  ^{\ast}R \right)  /_{\sim}=\left\{ o+S,r+S,t+S,w+S \right\}$ $\subset \mathcal{P}\left(  \mathcal{O}\right)$.

Let \textquotedblleft$\oplus$\textquotedblright\ and \textquotedblleft $ \odot $ \textquotedblright\  be operations on $R/_{w}S$, by using the Definition
\ref{Df7} and \ref{Df8}, as in Tables 9 and 10.%

\[%
\begin{array}
[c]{c|ccc}%
\oplus & r+S & t+S & w+S\\ \hline
r+S & t+S & w+S & o+S\\
t+S & w+S & o+S & r+S\\
w+S & o+S & r+S & t+S
\end{array}
\text{ \ \ \ \ \ \ \ \ \ \ \ \ \ \ \ \ }%
\begin{array}
[c]{c|ccc}%
\odot & r+S & t+S & w+S\\ \hline
r+S & t+S & o+S & t+S \\
t+S & o+S & o+S & o+S\\
w+S & t+S & o+S & t+S
\end{array}
\]%

\[
\text{\ \ \ \ \ \ }Table\text{ }9.\text{
\ \ \ \ \ \ \ \ \ \ \ \ \ \ \ \ \ \ \ \ \ \ \ \ \ \ \ \ \ \ \ \ \ \ \ \ \ \ \ \ }%
Table\text{ }10.
\]

It is enough to show that every element of $\left(  N_{1}\left(  B\right)
^{\ast}R\right)  /_{\sim}$ is also an element of $N_{1}\left(  B\right)
^{\ast}\left(  R/_{w}S\right)  $ in order to ensure $\left(  N_{r}\left(
B\right)  ^{\ast}R\right)  /_{\sim}\subseteq N_{r}\left(  B\right)  ^{\ast
}\left(  R/_{w}S\right)  $.%

\begin{align*}
\mathcal{Q}(R/_{w}S)  &  =\{\Phi(A)\mid A\in R/_{w}S\}\\
&  =\left\{  \Phi\left(  r+S\right)  ,\Phi\left(  t+S\right)  ,\Phi\left(
w+S\right)  \right\} \\
&  =\left\{  \left\{  \Phi\left(  r\right)  \right\}  ,\left\{  \Phi\left(
w\right)  ,\Phi\left(  r\right)  ,\Phi\left(  t\right)  \right\}  ,\left\{
\Phi\left(  t\right)  ,\Phi\left(  w\right)  \right\}  \right\} \\
&  =\{\left\{  \left(  \alpha_{1},\beta_{2}\right)  \right\}  ,\left\{
\left(  \alpha_{4},\beta_{1}\right)  ,\left(  \alpha_{2},\beta_{1}\right)
,\left(  \alpha_{1},\beta_{2}\right)  \right\}  ,\left\{  \left(  \alpha
_{1},\beta_{2}\right)  ,\left(  \alpha_{4},\beta_{1}\right)  \right\}
\}\text{.}%
\end{align*}

For $r+S\in R/_{w}S$, we get that%
\begin{align*}
\mathcal{Q}\left(  r+S\right)   &  =\left\{  \Phi\left(  r\right)  \right\}
=\left\{  \left(  \alpha_{1},\beta_{2}\right)  \right\}  \text{,}\\
\mathcal{Q}\left(  o+S\right)   &  =\left\{  \Phi\left(  r\right)
,\Phi\left(  w\right)  ,\Phi\left(  o\right)  \right\}  =\left\{  \left(
\alpha_{1},\beta_{2}\right)  ,\left(  \alpha_{4},\beta_{1}\right)  ,\left(
\alpha_{4},\beta_{1}\right)  \right\}  \text{.}%
\end{align*}

Since $\mathcal{Q}\left(  r+S\right)  \cap\mathcal{Q}\left(  o+S\right)
=\left\{  \left(  \alpha_{1},\beta_{2}\right)  \right\}  \neq\emptyset$, it
follows that $o+S\in\xi_{\Phi}\left(  r+S\right)  $. Hence $\xi_{\Phi}\left(
r+S\right)  \underset{\Phi}{\cap}R/_{w}S\neq\emptyset$ and $r+S,o+S\in
N_{1}\left(  B\right)  ^{\ast}\left(  R/_{w}S\right)  $ by Definition
\ref{Df9}.

For $t+S\in R/_{w}S$, $w+S$ we get that%
\begin{align*}
\mathcal{Q}\left(  t+S\right)   &  =\left\{  \left\{  \Phi\left(  w\right)
,\Phi\left(  r\right)  ,\Phi\left(  t\right)  \right\}  \right\}  =\left\{
\left(  \alpha_{4},\beta_{1}\right)  ,\left(  \alpha_{2},\beta_{1}\right)
,\left(  \alpha_{1},\beta_{2}\right)  \right\}  \text{,}\\
\mathcal{Q}\left(  w+S\right)   &  =\left\{  \left\{  \Phi\left(  t\right)
,\Phi\left(  w\right)  \right\}  \right\}  =\left\{  \left(  \alpha_{1}%
,\beta_{2}\right)  ,\left(  \alpha_{4},\beta_{1}\right)  \right\}  \text{.}%
\end{align*}

Since $\mathcal{Q}\left(  t+S\right)  \cap\mathcal{Q}\left(  t+S\right)
=\left\{  \left(  \alpha_{4},\beta_{1}\right)  ,\left(  \alpha_{2},\beta
_{1}\right)  ,\left(  \alpha_{1},\beta_{2}\right)  \right\}  \neq\emptyset$
and $\mathcal{Q}\left(  w+S\right)  \cap\mathcal{Q}\left(  w+S\right)
=\left\{  \left(  \alpha_{1},\beta_{2}\right)  ,\left(  \alpha_{4},\beta
_{1}\right)  \right\}  \neq\emptyset$, it follows that $t+S\in\xi_{\Phi
}\left(  t+S\right)  ,w+S\in\xi_{\Phi}\left(  w+S\right)  $. Hence $\xi_{\Phi
}\left(  t+S\right)  \underset{\Phi}{\cap}R/_{w}S\neq\emptyset,$ $\xi_{\Phi
}\left(  w+S\right)  \underset{\Phi}{\cap}R/_{w}S\neq\emptyset$ and
$t+S,w+S\in N_{1}\left(  B\right)  ^{\ast}\left(  R/_{w}S\right)  $ by
Definition \ref{Df9}.

Consequently, $\left(  N_{r}\left(  B\right)  ^{\ast}R\right)  /_{\sim_{L}%
}\subseteq N_{r}\left(  B\right)  ^{\ast}\left(  R/_{w}S\right)  $.

Thus, from the Theorem \ref{Th5}, $R/_{w}S$ is a nearness ring of all weak
cosets of $R$ by $S$ with the operations given by Tables 9 and 10.
\end{example}

\begin{definition}
\label{Df20}Let $R_{1},R_{2}\subset\mathcal{O}$ be two nearness rings and
$\eta$ a mapping from $N_{r}\left(  B\right)  ^{\ast}R_{1}$ onto $N_{r}\left(
B\right)  ^{\ast}R_{2}$. If $\eta\left(  x+y\right)  =\eta\left(  x\right)
+\eta\left(  y\right)  $ and $\eta\left(  x\cdot y\right)  =\eta\left(
x\right)  \cdot\eta\left(  y\right)  $ for all $x,y\in R_{1}$, then $\eta$ is
called a nearness ring homomorphism and also, $R_{1}$ is called near
homomorphic to $R_{2}$, denoted by $R_{1}\simeq_{n}R_{2}$.

A nearness ring homomorphism $\eta$ of $N_{r}\left(  B\right)  ^{\ast}R_{1}$
into $N_{r}\left(  B\right)  ^{\ast}R_{2}$ is called

(i) a nearness momomorphism if $\eta$ is one-one,

(ii) a nearness epimorphism if $\eta$ is onto $N_{r}\left(  B\right)  ^{\ast
}R_{2}$ and

(iii) a nearness isomorphism if $\eta$ is one-one and maps $N_{r}\left(
B\right)  ^{\ast}R_{1}$ onto $N_{r}\left(  B\right)  ^{\ast}R_{2}$.
\end{definition}

\begin{theorem}
\label{Th6}Let $R_{1},R_{2}$ be two nearness rings and $\eta$ a nearness
homomorphism of $N_{r}\left(  B\right)  ^{\ast}R_{1}$ into $N_{r}\left(
B\right)  ^{\ast}R_{2}$. Then the following properties hold.

(i) $\eta\left(  0_{R_{1}}\right)  =0_{R_{2}}$, where $0_{R_{2}}\in
N_{r}\left(  B\right)  ^{\ast}R_{2}$ is the nearness zero of $R_{2}$.

(ii) $\eta\left(  -x\right)  =-\eta\left(  x\right)  $ for all $x\in R_{1}$.
\end{theorem}

\begin{proof}
(i) Since $\eta$ is a nearness homomorphism, $\eta\left(  0_{R_{1}}\right)
\cdot\eta\left(  0_{R_{1}}\right)  =\eta\left(  0_{R_{1}}\cdot0_{R_{1}%
}\right)  =\eta\left(  0_{R_{1}}\right)  =\eta\left(  0_{R_{1}}\right)
\cdot0_{R_{2}}$. Thus we have that $\eta\left(  0_{R_{1}}\right)  =0_{R_{2}}$
by the Theorem \ref{Th02}.(iii).

(ii) Let $x\in R_{1}$. Then $\eta\left(  x\right)  \cdot\eta\left(  -x\right)
=\eta\left(  x-x\right)  =\eta\left(  0_{R_{1}}\right)  =0_{R_{2}}$. Similarly
we can obtain that $\eta\left(  -x\right)  \cdot\eta\left(  x\right)
=0_{R_{2}}$ for all $x\in R_{1}$. From Theorem \ref{Th01}.(ii), since
$\eta\left(  x\right)  $ has a unique inverse, $\eta\left(  -x\right)
=-\eta\left(  x\right)  $ for all $x\in R_{1}$.
\end{proof}

\begin{theorem}
\label{Th07}Let $R_{1},R_{2}$ be two nearness rings and $\eta$ a nearness
homomorphism of $N_{r}\left(  B\right)  ^{\ast}R_{1}$ into $N_{r}\left(
B\right)  ^{\ast}R_{2}$ and $N_{r}\left(  B\right)  ^{\ast}S$ a groupoid. Then
the following properties hold.

(i) If $S$ is a subnearness ring of nearness ring $R_{1}$ and $\eta\left(
N_{r}\left(  B\right)  ^{\ast}S\right)  =N_{r}\left(  B\right)  ^{\ast}%
\eta\left(  S\right)  $, then $\eta\left(  S\right)  =\left\{  \eta\left(
x\right)  :x\in S\right\}  $ is a subnearness ring of $R_{2}$.

(ii) If $S$ is a commutative subnearness ring $R_{1}$ and $\eta\left(
N_{r}\left(  B\right)  ^{\ast}S\right)  =N_{r}\left(  B\right)  ^{\ast}%
\eta\left(  S\right)  $, then $\eta\left(  S\right)  $ is a commutative
nearness ring of $R_{2}$.
\end{theorem}

\begin{proof}
(i) Let $S$ be a subnearness ring of nearness ring $R_{1}$. Then $0_{S}\in
N_{r}\left(  B\right)  ^{\ast}S$ and by Theorem \ref{Th6}.(i), $\eta\left(
0_{S}\right)  =0_{R_{2}}$, where $0_{R_{2}}\in N_{r}\left(  B\right)  ^{\ast
}R_{2}$. Thus, $0_{R_{2}}=\eta\left(  0_{S}\right)  \in\eta\left(
N_{r}\left(  B\right)  ^{\ast}S\right)  =N_{r}\left(  B\right)  ^{\ast}%
\eta\left(  S\right)  .$This means that $\eta\left(  S\right)  \neq\emptyset$.
Let $\eta\left(  x\right)  \in\eta\left(  S\right)  $, where $x\in S$. Since
$S$ is a subnearness ring of $R_{1}$, $-x\in N_{r}\left(  B\right)  ^{\ast}S$
for all $x\in S$. Thus $-\eta\left(  x\right)  =\eta\left(  -x\right)  \in
\eta\left(  N_{r}\left(  B\right)  ^{\ast}S\right)  =N_{r}\left(  B\right)
^{\ast}\eta\left(  S\right)  $ for all $\eta\left(  x\right)  \in\eta\left(
S\right)  $. Hence by Theorem \ref{Th1}, $\eta\left(  S\right)  $ is
subnearness ring of $R_{2}$.

(ii) Let $S$ be a commutative subnearness ring and $\eta\left(  x\right)
,\eta\left(  y\right)  \in\eta\left(  S\right)  $. We have that $\eta\left(
S\right)  $ is a subnearness ring of $R_{2}$ by (i), i.e., $\eta\left(
S\right)  $ is a nearness ring. Then $\eta\left(  x\right)  \cdot\eta\left(
y\right)  =\eta\left(  x\cdot y\right)  =\eta\left(  y\cdot x\right)
=\eta\left(  y\right)  \cdot\eta\left(  x\right)  $ for all $\eta\left(
x\right)  ,\eta\left(  y\right)  \in\eta\left(  R_{1}\right)  $. Hence
$\eta\left(  S\right)  $ is commutative subnearness ring of $R_{2}$.
\end{proof}

\begin{definition}
Let $R_{1},R_{2}$ be two nearness rings and $\eta$ be a nearness homomorphism
of $N_{r}\left(  B\right)  ^{\ast}R_{1}$ into $N_{r}\left(  B\right)  ^{\ast
}R_{2}$. The kernel of $\eta$, denoted by $Ker\eta$, is defined to be the set%
\[
Ker\eta=\left\{  x\in R_{1}:\eta\left(  x\right)  =0_{R_{2}}\right\}  \text{.}%
\]

\end{definition}

\begin{theorem}
\label{Th03}Let $R_{1},R_{2}$ be two nearness rings, $\eta$ a nearness
homomorphism of $N_{r}\left(  B\right)  ^{\ast}R_{1}$ into $N_{r}\left(
B\right)  ^{\ast}R_{2}$ and $N_{r}\left(  B\right)  ^{\ast}Ker\eta$ a groupoid
with binary operations \textquotedblleft$+$\textquotedblright\ and
\textquotedblleft$\cdot$\textquotedblright. Then $\emptyset\neq Ker\eta$ is a
nearness ideal of $R_{1}$.
\end{theorem}

\begin{proof}
Let $x,y\in Ker\eta$. Then $f\left(  x-y\right)  =f\left(  x\right)  -f\left(
y\right)  =0_{R_{2}}-0_{R_{2}}=0_{R_{2}}\in N_{r}\left(  B\right)  ^{\ast
}R_{2}$ and so $x-y\in N_{r}\left(  B\right)  ^{\ast}\left(  Ker\eta\right)
$. Let $r\in R_{1}$. Then $f\left(  r\cdot x\right)  =f\left(  r\right)  \cdot
f\left(  x\right)  =f\left(  r\right)  \cdot0_{R_{2}}=0_{R_{2}}\in
N_{r}\left(  B\right)  ^{\ast}R_{2}$ and so $r\cdot x\in N_{r}\left(
B\right)  ^{\ast}\left(  Ker\eta\right)  $. Similarly, $x\cdot r\in$
$N_{r}\left(  B\right)  ^{\ast}\left(  Ker\eta\right)  $. Hence, from
Definition \ref{Df6}, $Ker\eta$ is a nearness ideal of $R_{1}$.
\end{proof}

\begin{theorem}
Let $R$ be a nearness ring and $S$ a subnearness ring of $R$. Then the mapping
$\Pi:N_{r}\left(  B\right)  ^{\ast}R \rightarrow N_{r}\left(  B\right)  ^{\ast
}\left(  R/_{w}S\right)  $ defined by $\Pi\left(  x\right)  =x+S$ for all
$x\in N_{r}\left(  B\right)  ^{\ast}R$ is a nearness homomorphism.
\end{theorem}

\begin{proof}
From the definition of $\Pi$, $\Pi$ is a mapping from $N_{r}\left(  B\right)
^{\ast}R$ into $N_{r}\left(  B\right)  ^{\ast}\left(  R/_{w}S\right)  $. By
using the Definition \ref{Df8},%

\[
\Pi\left(  x+y\right)  =\left(  x+y\right)  +S=\left(  x+S\right)
\oplus\left(  y+S\right)  =\Pi\left(  x\right)  \oplus\Pi\left(  y\right)
\text{,}%
\]%
\[
\Pi\left(  x\cdot y\right)  =\left(  x\cdot y\right)  +S=\left(  x+S\right)
\odot\left(  y+S\right)  =\Pi\left(  x\right)  \odot\Pi\left(  y\right)
\]

for all $x,y\in R$. Thus $\Pi$ is a nearness homomorphism from Definition
\ref{Df20}.
\end{proof}

\begin{definition}
The near homomorphism $\Pi$ is called a nearness natural homomorphism from
$N_{r}\left(  B\right)  ^{\ast}R$ into $N_{r}\left(  B\right)  ^{\ast}\left(
R/_{w}S\right)  $.
\end{definition}

\begin{example}
From Example \ref{Ex4}, we consider the nearness ring of all weak cosets
$R/_{w}S$. Define%
\[%
\begin{tabular}
[c]{lll}%
$\Pi:N_{r}\left(  B\right)  ^{\ast}R$ & $\longrightarrow$ & $N_{r}\left(
B\right)  ^{\ast}\left(  R/_{w}S\right)  $\\
\multicolumn{1}{r}{$x$} & $\longmapsto$ & $\Pi\left(  x\right)  =x+S$%
\end{tabular}
\]

for all $x\in N_{r}\left(  B\right)  ^{\ast}R$. By using the Definitions
\ref{Df7} and \ref{Df8}, we have that%
\[
\Pi\left(  x+y\right)  =\left(  x+y\right)  +S=\left(  x+S\right)
\oplus\left(  y+S\right)  =\Pi\left(  x\right)  \oplus\Pi\left(  y\right)
\text{,}%
\]%
\[
\Pi\left(  x\cdot y\right)  =\left(  x\cdot y\right)  +S=\left(  x+S\right)
\odot\left(  y+S\right)  =\Pi\left(  x\right)  \odot\Pi\left(  y\right)
\]
for all $x,y\in R$. Hence, $\Pi\ $is a nearness natural homomorphism from
$N_{r}\left(  B\right)  ^{\ast}R$ into $N_{r}\left(  B\right)  ^{\ast}\left(
R/_{w}S\right)  $.
\end{example}

\begin{definition}
\label{Df22}Let $R_{1},R_{2}$ be two nearness rings and $S$ be a non-empty
subset of $R_{1}$. Let%
\[
\chi:N_{r}\left(  B\right)  ^{\ast}R_{1}\longrightarrow N_{r}\left(  B\right)
^{\ast}R_{2}%
\]

be a mapping and%
\[
\chi_{_{S}}=%
%TCIMACRO{\QTATOP{\chi}{{}}}%
%BeginExpansion
\genfrac{}{}{0pt}{1}{\chi}{{}}%
%EndExpansion%
%TCIMACRO{\QTATOPD{\vert}{.}{{}}{S}}%
%BeginExpansion
\genfrac{\vert}{.}{0pt}{1}{{}}{S}%
%EndExpansion
:S\longrightarrow N_{r}\left(  B\right)  ^{\ast}R_{2}%
\]

a restricted mapping. If $\chi\left(  x+y\right)  =\chi_{_{S}}\left(
x+y\right)  =\chi_{_{S}}\left(  x\right)  +\chi_{_{S}}\left(  y\right)
=\chi\left(  x\right)  +\chi\left(  y\right)  $ and $\chi\left(  x\cdot
y\right)  =\chi_{_{S}}\left(  x\cdot y\right)  =\chi_{_{S}}\left(  x\right)
\cdot\chi_{_{S}}\left(  y\right)  =\chi\left(  x\right)  \cdot\chi\left(
y\right)  $ for all $x,y\in S$, then $\chi$ is called a restricted nearness
homomorphism and also, $R_{1}$ is called restricted nearness homomorphic to
$R_{2}$, denoted by $R_{1}\simeq_{rn}R_{2}$.
\end{definition}

\begin{theorem}
Let $R_{1},R_{2}$ be two nearness rings and $\chi$ be a nearness homomorphism
from $N_{r}\left(  B\right)  ^{\ast}R_{1}$ into $N_{r}\left(  B\right)
^{\ast}R_{2}$ . Let $\left(  N_{r}\left(  B\right)  ^{\ast}Ker\chi,+\right)  $
and $\left(  N_{r}\left(  B\right)  ^{\ast}Ker\chi,\cdot\right)  $ be
groupoids and $\left(  N_{r}\left(  B\right)  ^{\ast}R_{1}\right)  /_{\sim}$
be a set of all weak cosets of $N_{r}\left(  B\right)  ^{\ast}R_{1}$ by
$Ker\chi$. If $\left(  N_{r}\left(  B\right)  ^{\ast}R_{1}\right)  /_{\sim
}\subseteq N_{r}\left(  B\right)  ^{\ast}\left(  R_{1}/_{w}Ker\chi\right)  $
and $N_{r}\left(  B\right)  ^{\ast}\chi\left(  R_{1}\right)  =\chi\left(
N_{r}\left(  B\right)  ^{\ast}R_{1}\right)  $, then%
\[
R_{1}/_{w}Ker\chi\simeq_{rn}\chi\left(  R_{1}\right)  \text{.}%
\]

\end{theorem}

\begin{proof}
Since $\left(  N_{r}\left(  B\right)  ^{\ast}Ker\chi,+\right)  $ and $\left(
N_{r}\left(  B\right)  ^{\ast}Ker\chi,\cdot\right)  $ are groupoids, from
Theorem \ref{Th03} $Ker\chi$ is a subnearness ring of $R_{1}$. Since $Ker\chi$
is a subnearness ring of $R_{1}$ and $\left(  N_{r}\left(  B\right)  ^{\ast
}R_{1}\right)  /_{\sim}$ $\subseteq N_{r}\left(  B\right)  ^{\ast}\left(
R_{1}/_{w}Ker\chi\right)  $, then $R_{1}/_{w}Ker\chi$ is a nearness ring of
all weak cosets of $R_{1}$ by $Ker\chi$ from Theorem \ref{Th5}. Since
$N_{r}\left(  B\right)  ^{\ast}\chi\left(  R_{1}\right)  =\chi\left(
N_{r}\left(  B\right)  ^{\ast}R_{1}\right)  $, $\chi\left(  R_{1}\right)  $ is
a subnearness ring of $R_{2}$. Define%
\small{
\[%
\begin{tabular}
[c]{lll}%
$\eta:N_{r}\left(  B\right)  ^{\ast}\left(  R_{1}/_{w}Ker\chi\right)  $ &
$\longrightarrow$ & $N_{r}\left(  B\right)  ^{\ast}\chi\left(  R_{1}\right)
$\\
\multicolumn{1}{r}{$A$} & $\longmapsto$ & $\eta(A)=\left\{
\begin{tabular}
[c]{ll}%
$\eta_{_{R_{1}/_{w}Ker\chi}}\left(  A\right)  $ & $,A\in\left(  N_{r}\left(
B\right)  ^{\ast}R_{1}\right)  /_{\sim}$\\
$e_{\chi\left(  R_{1}\right)  }$ & $,A\notin\left(  N_{r}\left(  B\right)
^{\ast}R_{1}\right)  /_{\sim}$%
\end{tabular}
\ \ \ \ \right.  $%
\end{tabular}
\ \ \ \
\]
}
\normalsize
where%
\[%
\begin{tabular}
[c]{lll}%
$\eta_{_{R_{1}/_{w}Ker\chi}}:%
%TCIMACRO{\QTATOP{\eta}{{}}}%
%BeginExpansion
\genfrac{}{}{0pt}{1}{\eta}{{}}%
%EndExpansion%
%TCIMACRO{\QTATOPD{\vert}{.}{{}}{R_{1}/_{w}Ker\chi}}%
%BeginExpansion
\genfrac{\vert}{.}{0pt}{1}{{}}{R_{1}/_{w}Ker\chi}%
%EndExpansion
$ & $\longrightarrow$ & $N_{r}\left(  B\right)  ^{\ast}\chi\left(
R_{1}\right)  $\\
\multicolumn{1}{r}{$x+Ker\chi$} & $\longmapsto$ & $\eta_{_{R_{1}/_{w}Ker\chi}%
}\left(  x+Ker\chi\right)  =\chi\left(  x\right)  $%
\end{tabular}
\ \ \ \
\]

for all $x+Ker\chi\in R_{1}/_{w}Ker\chi$.

Since%
\begin{align*}
x+Ker\chi &  =\left\{  x+k\mid k\in Ker\chi,x+k\in R_{1}\right\}  \cup\left\{
x\right\}  \text{,}\\
y+Ker\chi &  =\left\{  y+k^{\prime}\mid k^{\prime}\in Ker\chi,y+k^{\prime}\in
R_{1}\right\}  \cup\left\{  y\right\}  \text{,}%
\end{align*}

and the mapping $\chi$ is a nearness homomorphism,%
\[%
\begin{tabular}
[c]{rl}
& $x+Ker\chi=y+Ker\chi$\\
$\Rightarrow$ & $x\in y+Ker\chi$\\
$\Rightarrow$ & $x\in\left\{  y+k^{\prime}\mid k^{\prime}\in Ker\chi
,y+k^{\prime}\in R_{1}\right\}  $ or $x\in\left\{  y\right\}  $\\
$\Rightarrow$ & $x=y+k^{\prime},$ $k^{\prime}\in Ker\chi,$ $y+k^{\prime}\in
R_{1}$ or $x=y$\\
$\Rightarrow$ & $-y+x=\left(  -y+y\right)  +k^{\prime},$ $k^{\prime}\in
Ker\chi$ or $\chi\left(  x\right)  =\chi\left(  y\right)  $\\
$\Rightarrow$ & $-y+x=k^{\prime},$ $k^{\prime}\in Ker\chi$\\
$\Rightarrow$ & $-y+x\in Ker\chi$\\
$\Rightarrow$ & $\chi\left(  -y+x\right)  =e_{\chi\left(  R_{1}\right)  }$\\
$\Rightarrow$ & $\chi\left(  -y\right)  +\chi\left(  x\right)  =e_{\chi\left(
R_{1}\right)  }$\\
$\Rightarrow$ & $-\chi\left(  y\right)  +\chi\left(  x\right)  =e_{\chi\left(
R_{1}\right)  }$\\
$\Rightarrow$ & $\chi\left(  x\right)  =\chi\left(  y\right)  $\\
$\Rightarrow$ & $\eta_{_{R_{1}/_{w}Ker\chi}}\left(  x+Ker\chi\right)
=\eta_{_{R_{1}/_{w}Ker\chi}}\left(  y+Ker\chi\right)  $%
\end{tabular}
\ \ \ \
\]

Therefore $\eta_{_{R_{1}/_{w}Ker\chi}}$ is well defined.

For $A,B\in N_{r}\left(  B\right)  ^{\ast}\left(  R_{1}/_{w}Ker\chi\right)$,
we suppose that $A=B$. Since the mapping $\eta_{_{R_{1}/_{w}Ker\chi}}$ is well
defined,%
\[%
\begin{tabular}
[c]{ll}%
$\eta \left(  A\right) $ & $=\left\{
\begin{tabular}
[c]{ll}%
$\eta_{_{R_{1}/_{w}Ker\chi}}\left(  A\right) $ & ,$A\in\left(  N_{r}\left(
B\right)  ^{\ast}R_{1}\right)  /_{\sim}$\\
$e_{\chi\left(  R_{1}\right)  } $& $,A\notin\left(  N_{r}\left(  B\right)
^{\ast}R_{1}\right)  /_{\sim}$%
\end{tabular}
\ \ \ \ \right. $\\
& \\
& $=\left\{
\begin{tabular}
[c]{ll}%
$\eta_{_{R_{1}/_{w}Ker\chi}}\left(  B\right)$   & ,$B\in\left(  N_{r}\left(
B\right)  ^{\ast}R_{1}\right)  /_{\sim}$\\
$e_{\chi\left(  R_{1}\right) } $ & $,B\notin\left(  N_{r}\left(  B\right)
^{\ast}R_{1}\right)  /_{\sim}$%
\end{tabular}
\ \ \ \ \right. $ \\
& \\
& $=\eta\left(  B\right) $.
\end{tabular}
\ \ \ \
\]

Consequently $\eta$ is well defined.

For all $x+Ker\chi,y+Ker\chi\in R_{1}/_{w}Ker\chi\subset N_{r}\left(
B\right)  ^{\ast}\left(  R_{1}/_{w}Ker\chi\right)  $,%
\[%
\begin{tabular}
[c]{ll}
& $\eta\left(  \left(  x+Ker\chi\right)  \oplus\left(  y+Ker\chi\right)
\right)  $\\
$=$ & $\eta_{_{R_{1}/_{w}Ker\chi}}\left(  \left(  x+Ker\chi\right)
\oplus\left(  y+Ker\chi\right)  \right)  $\\
$=$ & $\eta_{_{R_{1}/_{w}Ker\chi}}\left(  \left(  x+y\right)  +Ker\chi\right)
$\\
$=$ & $\chi\left(  x+y\right)  $\\
$=$ & $\chi\left(  x\right)  +\chi\left(  y\right)  $\\
$=$ & $\eta_{_{R_{1}/_{w}Ker\chi}}\left(  x+Ker\chi\right)  +\eta
_{_{R_{1}/_{w}Ker\chi}}\left(  y+Ker\chi\right)  $\\
$=$ & $\eta\left(  x+Ker\chi\right)  +\eta\left(  y+Ker\chi\right)  $.
\end{tabular}
\ \ \ \
\]

and%

\[%
\begin{tabular}
[c]{ll}
& $\eta\left(  \left(  x+Ker\chi\right)  \odot\left(  y+Ker\chi\right)
\right)  $\\
$=$ & $\eta_{_{R_{1}/_{w}Ker\chi}}\left(  \left(  x+Ker\chi\right)
\odot\left(  y+Ker\chi\right)  \right)  $\\
$=$ & $\chi_{_{R_{1}/_{w}Ker\chi}}\left(  \left(  x\cdot y\right)
+Ker\chi\right)  $\\
$=$ & $\chi\left(  x\cdot y\right)  $\\
$=$ & $\chi\left(  x\right)  \cdot\chi\left(  y\right)  $\\
$=$ & $\eta_{_{R_{1}/_{w}Ker\chi}}\left(  x+Ker\chi\right)  \cdot\eta
_{_{R_{1}/_{w}Ker\chi}}\left(  y+Ker\chi\right)  $\\
$=$ & $\eta\left(  x+Ker\chi\right)  \cdot\eta\left(  y+Ker\chi\right)  $.
\end{tabular}
\ \
\]

Therefore $\eta$ is a restricted nearness homomorphism by Definition
\ref{Df22}. Hence, $R_{1}/_{w}Ker\chi\simeq_{rn}\chi\left(  R_{1}\right)  $.
\end{proof}

\end{document}